\documentclass[]{article}

\usepackage{amsmath}
\usepackage{amssymb}
\usepackage{amsthm}

\usepackage{times}
\usepackage[utf8]{inputenc}
\usepackage[T1]{fontenc}

\usepackage{geometry}
\newtheorem{lem}{Lemma}
\newtheorem{cor}[lem]{Corollary}
\newtheorem{thm}[lem]{Theorem}
\newtheorem{thm*}{Theorem}

\newtheorem{prop}[lem]{Proposition}

\theoremstyle{definition}
\newtheorem{defin}[lem]{Definition}

\newcommand{\slice}{\operatorname{Slice}_{\omega_1}}
\newcommand{\slicetwo}{\operatorname{Slice}_{\omega_2}}
\newcommand{\slicekappa}{\operatorname{Slice}_{\kappa}}
\newcommand{\slicethree}{\operatorname{Slice}_{\omega_3}}
\newcommand{\ch}{\operatorname{CH}}

\newcommand{\ma}{\operatorname{MA}}
\newcommand{\zfc}{\operatorname{ZFC}}
\newcommand{\ac}{\operatorname{AC}}
\newcommand{\baire}{\omega^{\omega}}
\newcommand{\anal}{\Sigma^1_1}
\newcommand{\coanal}{\Pi^1_1}
\newcommand{\sym}{\operatorname{c(Sym(\omega))}}

\newcommand{\dom}{\operatorname{dom}}
\newcommand{\rg}{\operatorname{rg}}

\title{The Slicing Axioms}
\author{Ziemowit Kostana\footnote{Research of Z. Kostana was supported by the GAČR project EXPRO 20-31529X and
RVO: 67985840.},\\
Institute of Mathematics of the Czech Academy of Sciences, Czech Republic\\
University of Warsaw, Poland\\
Saharon Shelah\footnote{Research of second author partially supported by Israel Science Foundation (ISF) grant no: 1838/19.
Paper number 1210 on the publication list.}
\footnote{Research of both authors partially supported by NSF grant no: DMS 1833363.}\\
Einstein Institute of Mathematics, the Hebrew University of Jerusalem, Israel\\
Department of Mathematics, Rutgers University, Piscataway, USA}

\begin{document}

\maketitle

\begin{abstract}
	We introduce the family of axioms, denoted $\slicekappa$, that claim the existence of strictly increasing decompositions of the form
	$$2^{\delta}=\bigcup_{\alpha<\kappa} 2^{\delta}\cap M_\alpha,$$
	where $\delta<\kappa$, and $\{M_\alpha|\; \alpha<\kappa\}$ is a $\subseteq$-increasing sequence of transitive models of set theory. We study compatibility of these axioms with versions of Martin's Axiom, and in particular show that $\slice$ is compatible only with some very weak form of $\ma$.
\end{abstract}


\section{Introduction}
\subsection{How "compact" is the real line?}

We introduce and study a family of axioms $\slicekappa$ for cardinal numbers $\kappa$. The axiom $\slicekappa$ basically claims that there exists an increasing sequence $\{M_\alpha|\; \alpha<\kappa\}$ of transitive models of $\zfc$, that decomposes the sets $2^{\delta}$, for $\delta<\kappa$, into increasing unions $$2^{\delta}=\bigcup_{\alpha<\kappa} 2^{\delta}\cap M_\alpha.$$ 
Our initial motivation was to find a single model of Martin's Axiom, which doesn't satisfy typical consequences of $PFA$. This was in turn motivated by the following intuition:

	\begin{center}If the universe is sufficiently complete, in the sense that it has many generic filters, then any transitive submodel containing enough reals, contain all the reals. 
	\end{center}

This intuition is supported for example by the following result:

\begin{thm}[Thm. 8.6, \cite{continuum}]
	If $MM$ holds, then any inner model with correct $\omega_2$ contains all reals.
\end{thm}

The conclusion is quite strong, so it makes sense to ask what is left if we weaken $\operatorname{MM}$ to $\ma_{\omega_1}$. This motivated us to formulate the axiom $\slice$, which turned out to be inconsistent with $\ma_{\omega_1}$. The main results of this paper are the following

\begin{thm}[Thm. \ref{martinsaxiom}]
	$\slice \implies \lnot \operatorname{MA}_{\omega_1}(\text{$\sigma$-centred})$.
\end{thm}

\begin{thm}[Thm. \ref{mainthm}]
	If $\kappa$ is a regular cardinal such that $\kappa^\omega=\kappa$, then the following theory is consistent 
	$$\zfc + \ma(\text{Suslin}) + \slice + ``2^\omega=\kappa".$$
\end{thm}

\begin{thm}[Thm. \ref{mainthm2}]
Assume that $\omega<\kappa\le\theta$ are regular cardinals, such that $\theta^{<\kappa}=\theta$. Then the following theory is consistent for any cardinal $\lambda<\kappa$:
$$\zfc+\ma_{\lambda}+\slicekappa + ``2^\omega=\theta".$$
\end{thm}

The first of these results provides another argument in favor of the informal claim from the beginning.
The class of Suslin forcings is a class of c.c.c. forcings, which admit simple (analytic) definitions (see Definition \ref{Suslindef}). This class is more extensively described in \cite{b-j}. Martin's Axiom for this class is a considerable weakening of the full $\ma$.

\begin{thm}[\cite{292}] $\ma(\text{Suslin})$ implies each of the following:
	\begin{enumerate}
		\item $\operatorname{Add}(\mathcal{N})=2^\omega$,
		\item $\operatorname{Add}(\mathcal{SN})=2^\omega$,
		\item $2^\omega$ is regular,
		\item each MAD family of subsets of $\omega$ has size $2^\omega$.
	\end{enumerate}
\end{thm}

It follows from 1. that all cardinal characteristics in the Cicho\'n's diagram have value $2^\omega$. $\mathcal{SN}$ stands for the class of strong measure zero sets.

\begin{thm}[\cite{292}] $\ma(\text{Suslin})$ does not imply any of the following:
	\begin{enumerate}
		\item $\mathfrak{t}=2^\omega$,
		\item $\mathfrak{s}=2^\omega$,
		\item $\forall\; {\kappa<2^\omega}\quad 2^\kappa=2^\omega$,
		\item there is no Suslin tree.
	\end{enumerate}
\end{thm}

For an elaborated discussion of cardinal invariants of the continuum we refer the reader to \cite{blass}. Finally, it should be noted that our axioms $\slice$ is similar to the axioms $\diamondsuit^{\operatorname{Cohen}}$ introduced recently in \cite{diamondcohen}.
 
\subsection{Preliminaries} 
All non-standard notions are introduced in the subsequent sections. By \emph{reals} we mean elements of the sets $\baire$, $2^\omega$, or seldom $\mathbb{R}$. We take liberty to freely identify Borel functions with their Borel codes, so whenever we claim that
$$f \in M \models \zfc$$
for some Borel function $f\subseteq 2^\omega \times 2^\omega$, it should be understood that it is the Borel code of $f$ that belongs to $M$ (so we don't bother if, for instance, $\dom{f} \not\subseteq M$). For a cardinal $\kappa$,  $\mathbb{C}_\kappa=\{p:\dom(p)\rightarrow 2|\; \dom(p)\in [\kappa]^{<\omega}\}$, and $\mathbb{C}=\mathbb{C}_\omega$.

When we write $\mathbb{P}=\{\mathbb{P}_\alpha \ast \dot{\mathbb{Q}}_\alpha|\; \alpha<\theta \}$ for a finite support iteration of forcings, we sometimes denote by $\mathbb{P}$ the final step of the iteration, that is $\mathbb{P}=\mathbb{P}_\theta$. When working with infinite iterations, we assume that $\mathbb{P}_0$ is the trivial forcing. A function $i:\mathbb{P}_0 \hookrightarrow \mathbb{P}_1$ is a \emph{complete embedding} if the following assertions hold:
\begin{enumerate}
	\item $\forall\;{p,q \in \mathbb{P}_0}\quad p_0\le p_1 \implies i(p_0)\le i(p_1)$,
	\item $\forall\;{p,q \in \mathbb{P}_0}\quad p_0\bot p_1 \implies i(p_0)\bot i(p_1)$,
	\item If $\mathcal{A} \subseteq \mathbb{P}_0$ is a maximal antichain, then $i[\mathcal{A}]\subseteq \mathbb{P}_1$ is a maximal antichain.
\end{enumerate}
We write $\mathbb{P}_0 \lessdot\mathbb{P}_1$ if $\mathbb{P}_0\subseteq \mathbb{P}_1$ and the inclusion is a complete embedding. The importance of complete embeddings comes from the closely related notion of \emph{quotient forcing}. If $\mathbb{P}_0\lessdot \mathbb{P}_1$, and $G\subseteq \mathbb{P}_0$ is a filter generic over $V$, then $V[G]$ contains the quotient forcing $\mathbb{P}_1/G\subseteq \mathbb{P}_1$, consisting of all conditions $p \in \mathbb{P}_1$ that are compatible with every element of $G$. The crucial property of this notion is that forcing with $\mathbb{P}_1$ is equivalent to forcing with a two-step iteration: $\mathbb{P}_0\ast \mathbb{P}_1/\dot{G}$ (see \cite{kunen}, p.244).

We will be frequently using the following two observations.

\begin{prop} \label{observation}
	If $V$ is a countable transitive model of $\zfc$, $\mathbb{P}_0,\mathbb{P}_1 \in V$, and $\mathbb{P}_0 \subseteq\mathbb{P}_1$ is an inclusion of partial orders, then the following conditions are equivalent:
	\begin{enumerate}
		\item $\mathbb{P}_0 \lessdot\mathbb{P}_1$,
		\item If a filter $G\subseteq \mathbb{P}_1$ is $\mathbb{P}_1$-generic over $V$ then
		$G \cap \mathbb{P}_0$ is $\mathbb{P}_0$-generic over $V$.
	\end{enumerate}
\end{prop}
	
\begin{prop} \label{aswellas}
	Let $\mathbb{P} \lessdot \mathbb{S}$ be any forcing notions, and fix $p \in \mathbb{P}$. Let $\dot{q}$, $\dot{r}$ be $\mathbb{P}$-names, and finally let $\sigma(-,-)$ be a formula with parameters in the ground model, which is also absolute between transitive models of set theory (for example a $\anal$ formula in the language of arithmetic, or a bounded formula in the language of set theory). Then
	$$p \Vdash_{\mathbb{S}} \sigma(\dot{q},\dot{r}) \iff p \Vdash_{\mathbb{P}} \sigma(\dot{q},\dot{r}).$$
\end{prop}
	
\begin{proof}
	In the direction from left to right, if $p\in G \subseteq \mathbb{P}$ is generic over the ground model $V$, then we can extend $G$ to a generic filter $G'\subseteq \mathbb{S}$. Notice that $\dot{q}[G]=\dot{q}[G']$, and $\dot{r}[G]=\dot{r}[G']$.  By the absoluteness of $\sigma(-,-)$ we have
	$$V[G]\models \sigma(\dot{q}[G],\dot{r}[G]) \iff V[G']\models \sigma(\dot{q}[G'],\dot{r}[G']).$$
	In the other direction, we proceed in a similar way, using the fact that for any generic filter $G\subseteq \mathbb{S}$, the intersection $G\cap\mathbb{P}$ is $\mathbb{P}$-generic.$\qedhere$
\end{proof}

\section{The Slicing Axioms}

\begin{defin}
Let $\kappa$ be any uncountable cardinal. We will say that $\slicekappa$ holds if there exists a $\subseteq$-increasing sequence of transitive classes (not necessarily proper) $\{M_\alpha|\; \alpha<\kappa\}$, such that the following conditions are satisfied

\begin{enumerate}
	\item $\forall\;{\alpha<\kappa}\quad (M_\alpha;\in) \models \zfc$
	\item $\forall\;{\alpha < \kappa}\quad \omega_1^{M_\alpha}=\omega_1$,
	\item $\forall\;{\delta \in[\omega,\kappa)}\quad 2^{\delta}= \displaystyle{ \bigcup_{\alpha <\kappa} 2^{\delta} \cap M_\alpha}$,
	\item $\forall\;{\alpha<\beta<\kappa}\;\forall\;{\delta \in[\omega,\kappa)}\quad 2^{\delta} \cap M_\alpha \subsetneq 2^{\delta}\cap M_\beta$.
\end{enumerate}	
\end{defin}

Let us comment briefly on the choice of the definition. The reader uncomfortable with the idea of working with models of $\zfc$ can replace $\zfc$ by some sufficiently, large finite fragment of it. Such a formal weakening has no impact on the results of the paper. Moreover, it seems very natural to consider a sequence of models the length $\kappa>\omega_1$, which is "slicing" just $2^\omega$, instead $2^\delta$, for any $\delta<\kappa$, as in condition $3$. We opted for this definition mainly because, as we will see later, slicing $2^\delta$ for uncountable $\delta$ ensures the slicing of $2^\omega$ in generic extensions. Moreover, in the main cases of our interest, the models $M_\alpha$ will be intermediate forcing extensions, and then condition $3.$ will hold. As for the condition $4.$ it is equivalent to a formally weaker

\begin{enumerate}
\item[4a.]$$\forall \; \alpha<\beta<\kappa \quad 2^\omega \cap M_\alpha \subsetneq 2^\omega \cap M_\beta.$$
\end{enumerate}

This is because, having a real $x \in 2^\omega \in (M_\beta \setminus M_\alpha)$, we can add a trail of zeros to get a sequence of any given length\footnote{of any given length belonging to $M_\beta$, to be precise. However, condition $3.$ ensures that the models from the sequence know every ordinal below $\kappa$, from some point on.} to $M_\beta \setminus M_\alpha$. For the same reason, if $\kappa=\lambda^+$, we can just require
\begin{enumerate}
	\item[4b.]$$\forall \; \alpha<\beta<\kappa \quad 2^\lambda \cap M_\alpha \subsetneq 2^\lambda \cap M_\beta.$$
\end{enumerate}

The axiom $\slicekappa$ is outright false for any singular $\kappa$.

\begin{prop}
	If $\slicekappa$ holds, then $\kappa$ is regular, and $\kappa \le 2^\omega$.
\end{prop}

\begin{proof}
	If $\delta=\operatorname{cof}{\kappa}<\kappa$, then $\slicekappa$ gives us a nontrivial decomposition of the form
	$$2^{\delta}= \bigcup_{\alpha <\kappa} 2^{\delta} \cap M_\alpha,$$
	and after passing to a cofinal sequence, we also have a decomposition
	$$2^{\delta}= \bigcup_{\alpha <\delta} 2^{\delta} \cap M_{\gamma_\alpha}.$$
	For each $\alpha < \delta$ we pick $s_\alpha \in 2^\delta \setminus M_{\gamma_\alpha}$. Using a definable bijection between $\delta$ and $\delta \times \delta$, we see that the sequence $(s_\alpha)_{\alpha<\delta}$ must belong so some model $M_\alpha$. But this contradicts the choice of $s_\alpha$. The inequality $\kappa \le 2^\omega$ follows outright from the condition $4.$ in the definition of $\slicekappa$, applied to $\delta=\omega$.$\qedhere$
\end{proof}

The most important of the slicing axioms is perhaps $\slice$, since it claims that the real line can be decomposed into an increasing union of $\omega_1$ many sets, which belong to bigger and bigger models. The fact that $\ma_{\omega_1}$ is inconsistent with $\slice$ shows, that the Martin's Axiom on $\omega_1$ imposes certain compactness on the real line. 

\section{Slicing the real line}

We begin with showing that Martin's Axiom on $\omega_1$ is not compatible with $\slice$.

\begin{thm}\label{martinsaxiom}
	$\slice \implies \lnot \operatorname{MA}_{\omega_1}(\text{$\sigma$-centred})$.
\end{thm}

In the proof we will utilize a known result from \cite{Q-set}. Recall that a set $A \subseteq \mathbb{R}$ is a \emph{$Q$-set}, if each subset of $A$ is a relative $F_\sigma$, i.e. for each $B\subseteq A$ there exists an $F_\sigma$ subset $F\subseteq \mathbb{R}$ such that $A\cap F=B$.

\begin{thm}[\cite{Q-set}]
	$\operatorname{MA}_{\omega_1}(\sigma\text{-centred})$ implies that each set of cardinality $\omega_1$ is a $Q$-set.
\end{thm}

\begin{proof}[Proof of Theorem \ref{martinsaxiom}]
	 Assume that $\operatorname{MA}_{\omega_1}(\text{$\sigma$-centred})$ holds, and $(M_\alpha)_{\alpha<\omega_1}$ is a sequence of models witnessing $\slice$. $M_0 \models ``\text{$2^\omega$ is uncountable}"$, so there exists a sequence of pairwise distinct reals $X=\{x_\alpha|\; \alpha<\omega_1\}\in M_0$ (note that this sequence is \emph{really} of the length $\omega_1$). Let $f:\omega_1 \hookrightarrow 2^\omega$ be a function such that $\forall \alpha<\omega_1\; f(\alpha) \notin M_\alpha$. We will obtain a contradiction, by showing that there exists some $\eta <\omega_1$, for which $\rg(f)\subseteq M_\eta$.
	\par For every natural number $m$, let $A_m=\{x_\alpha|\;f(\alpha)(m)=1 \}=X\cap F_m$, where $F_m$ is an $F_\sigma$ subset of reals. Since the sequence $(F_m)_{m<\omega}$ can be coded by a real, clearly it belongs to some model $M_\eta$. It is enough to show that using this sequence and $X$ we can give a definition of $\rg(f)$. But 
	\begin{equation*}\rg(f)=\{x\in 2^\omega|\;  \exists \alpha<\omega_1\; \forall m<\omega\quad x_\alpha \in F_m \iff x(m)=1\}. \hfill \qedhere \end{equation*} 
\end{proof}

It is compatible with any value of $2^\omega$ that $\slice$ holds.

\begin{prop} \label{prodprop}
	Let $\mathbb{P}$ be any finite support product of productively c.c.c. forcings adding reals, and of the length at least $\kappa$, where $\kappa$ is regular. Then $\mathbb{P} \Vdash \slicekappa$.
\end{prop}

\begin{proof}
	The forcing $\mathbb{P}$ under discussion is of the form
	$$\mathbb{P}=\prod_{i \in I\times \kappa} \mathbb{P}_i,$$
	where each $\mathbb{P}_i$ adds some real number, and $|I|\ge \kappa$. For each $\alpha < \kappa$, the product 
	$\displaystyle{\prod_{i \in I\times \alpha} \mathbb{P}_i}$ can be identified with a complete suborder of $\mathbb{P}$.	If $G\subseteq \mathbb{P}$ is generic over some model $V$, then $\slicekappa$ is witnessed by the sequence 
	\begin{equation*}
		M_\alpha=V[G\cap \prod_{i \in I\times \alpha} \mathbb{P}_i].\qedhere
	\end{equation*}
\end{proof}
\begin{cor}\label{cohen}
	For each uncountable, regular $\kappa$, $\mathbb{C}_\kappa \Vdash \slicekappa$.
\end{cor}

Recall that a set of reals is called \emph{$\omega_1$-dense}, if each nonempty open interval in this set has size $\omega_1$. The following was proved by Baumgartner in \cite{baum}. 

\begin{thm}[\cite{baum}]
	It is consistent with $\ma_{\omega_1}$, that all $\omega_1$-dense subsets of reals are order-isomorphic. In particular, each $\omega_1$-dense set of reals has a non-trivial order-automorphism.
\end{thm}

The natural question whether this assertion follows from $\ma_{\omega_1}$ was resolved by Avraham and the second author in \cite{a-s}.

\begin{thm}[\cite{a-s}]
	It is consistent with $\ma_{\omega_1}$, that there exists a rigid $\omega_1$-dense real order type.
\end{thm}

This is also an easy consequence of $\slice$.

\begin{thm}
	$\slice$ implies that there is an $\omega_1$-dense rigid subset of the real line.
\end{thm}

\begin{proof}
	Let $(M_\alpha)_{\alpha<\omega_1}$ be a sequence witnessing $\slice$. For each $\alpha$, we choose $$x_\alpha \in \mathbb{R} \cap (M_\alpha \setminus \bigcup_{\beta<\alpha}M_\beta).$$
	We can easily arrange the construction, so that we hit each open interval $\omega_1$-many times. The set $X=\{x_\alpha|\;\alpha<\omega_1\}$ is $\omega_1$-dense, and it remains to prove, that it is also rigid. Suppose that $f:X\rightarrow X$ is an order isomorphism. $f$ extends uniquely to a continuous function $f':\mathbb{R} \rightarrow \mathbb{R}$, and each such function can be coded by a real number. Therefore there is some $\eta<\omega_1$, such that $f' \in M_\eta$. Now, for any $\xi>\eta$, it is not possible that $f(x_\eta)=x_\xi$, because it would mean $x_\xi \in M_\eta$, contrary to the choice of $x_\xi$. But, likewise, it is not possible that $f^{-1}(x_\eta)=x_\xi$. The conclusion is that for all $\xi>\eta$, $f(x_\xi)=x_\xi$. But this means that $f$ is identity on a dense set, and therefore everywhere.$\qedhere$
\end{proof}

\section{Slicing the real line while preserving $\ma$(Suslin)}

We are going to show that $\slice$ is consistent with a version of Martin's Axiom which takes into account only partial orders representable as analytic sets (see \cite{b-j}, Ch. 3.6, or \cite{292}).


\begin{defin} \label{Suslindef}
	A partial order $(\mathbb{P},\le)$ has a \emph{Suslin definition} if $\mathbb{P} \in \anal(\baire)$, and both ordering and incompatibility relations in $\mathbb{P}$ are analytic relations on $\baire$. $\mathbb{P}$ is \emph{Suslin} if it has a Suslin definition and is c.c.c.
\end{defin}

The rest of the Section is devoted to the proof of the next theorem.

\begin{thm} \label{mainthm}
	If $\kappa$ is a regular cardinal such that $\kappa^\omega=\kappa$, then the following theory is consistent 
	$$\zfc + \ma(\text{Suslin}) + \slice + ``2^\omega=\kappa".$$
\end{thm}

Let $\psi(-,-,-,-)$ be a universal analytic formula, i.e. a $\Sigma_1^1$ formula with the property that for each analytic set $P\subseteq \baire \times \baire \times \baire$ there exists $r \in \baire$ such that
	$$P=\{x \in \omega^\omega\times\omega^\omega\times\omega^\omega|\; \psi(x,r)\}.$$
We want to use $\psi$ to add generic filters to all possible Suslin forcings. We will say that $\psi(-,-,-,\dot{r}_\alpha)$ \emph{defines} $\dot{\mathbb{Q}}_\alpha$ if $\dot{r}_\alpha$ is a $\mathbb{P}_\alpha$-name for a real and $\mathbb{P}_\alpha$ forces each of the following

$$\dot{\mathbb{Q}}_\alpha \text{ is a separative partial order with the greatest element } 0,$$
$$\psi(x,1,1,\dot{r}_\alpha)\iff x \in \dot{\mathbb{Q}}_\alpha,$$
$$\psi(x,y,2,\dot{r}_\alpha)\iff x \le_{\dot{\mathbb{Q}}_\alpha} y,$$
$$\psi(x,y,3,\dot{r}_\alpha)\iff x \bot_{\dot{\mathbb{Q}}_\alpha} y.$$

We will write $\psi^{\in}(x,z)$ for $\psi(x,1,1,z)$, $\psi^{\bot}(x,y,z)$ for $\psi(x,y,3,z)$, and $\psi^{\le}(x,y,z)$ for $\psi(x,y,2,z)$.

We are going to iterate all Suslin forcings, each of them cofinally many times. More precisely, we define by induction a finite support iteration $\{\mathbb{P}_\alpha\ast \dot{\mathbb{Q}}_\alpha|\; \alpha<\kappa\}$:

\begin{itemize}
	\item $\mathbb{P}_0=\{0\}$,
	\item $\mathbb{P}_\alpha \Vdash ``\dot{\mathbb{Q}}_\alpha=\{x \in \baire|\;\psi(x,\dot{r}_\alpha) \} \text{ if this formula defines a Suslin forcing;\; \\
	else } \dot{\mathbb{Q}}=\{0\}"$,
\end{itemize}

The variable $\dot{r}_\alpha$ ranges over all reals, and all possible names for reals, each of them cofinally many times. In order to iterate through all possible parameters using a suitable bookkeeping, we introduce the class of \emph{simple} conditions, following \cite{b-j}.

\begin{defin} By the induction on $\alpha$, we define \emph{simple} conditions in $\mathbb{P}_\alpha$.
	\begin{itemize}
		\item $\alpha=0.$ $\mathbb{P}_0=\{0\}$, and we declare $0$ to be simple.
		\item $\alpha+1.$ $(p,\dot{q})\in \mathbb{P}_{\alpha+1}$ is simple if $p\in \mathbb{P}_\alpha$ is simple and 
		$$\dot{q}=\{(m,n,p^m_n)|\; m,n < \omega,\, p^m_n \in \mathbb{P}_\alpha\},$$
		 where each $p^m_n$ is a simple condition in $\mathbb{P}_\alpha$. (For each $m\in \omega$, the set $\{p^m_n|\; n<\omega\}$ is a maximal antichain deciding $\dot{q}(m)$, i.e. $p^m_n \Vdash \dot{q}(m)=n$.)
		\item $\lim{\alpha}.$ $p \in \mathbb{P}_\alpha$ is simple if for each $\beta<\alpha$, $p \restriction \beta \in \mathbb{P}_\beta$ is simple.
		\end{itemize}
\end{defin}

It is straightforward to check by induction, that the set of simple $\mathbb{P}_\alpha$-conditions is dense in $\mathbb{P}_\alpha$, and that each $\mathbb{P}_\alpha$ has at most $\kappa$ many names for reals (if we restrict to names with simple conditions). We declare the forcings $\mathbb{P}_\alpha$ to consist only of simple conditions, so formally we write
$$\mathbb{P}_{\alpha+1}=\{ (p,\dot{q}) \in \mathbb{P}_\alpha \ast \dot{\mathbb{Q}}_\alpha|\; (p,\dot{q}) \text{ is simple}\}.$$

\begin{prop} \label{MAproof}
	If $\kappa$ is an uncountable regular cardinal such that $\kappa^\omega=\kappa$, then 
	$$\mathbb{P}_\kappa \Vdash \ma(\text{Suslin}) + ``2^\omega=\kappa".$$
\end{prop}

\begin{proof}
	Let us denote by $W_\alpha$ the corresponding extensions of $V$ by $\mathbb{P}_\alpha$. Let $(S,\le)$ be a Suslin forcing in $W_\kappa$. Assume $S$ is defined by the formula $\psi(-,r)$. We fix a family $\{A_\gamma|\; \gamma<\lambda\}$ of maximal antichains in $S$, where $\lambda<\kappa$. Of course the formula $\psi(-,r)$ defines different sets in different models of set theory, so following the common custom we will denote by $S^N$ the interpretation of $S$ in the model $N$, i.e.
	$$S^N=\{x\in \baire \cap N|\; N\models \psi^{\in}(x,r) \}.$$
	
	Notice that the family $\{A_\gamma|\; \gamma<\lambda\}$ is a function from $\lambda$ to $[\baire]^\omega$, and so is added in some intermediate step if the iteration. Let us fix an ordinal $\delta<\kappa$ such that $\{A_\gamma|\; \gamma<\lambda\} \in W_\delta$, and $\mathbb{P}_\delta \Vdash \dot{r}_\delta = r$. Now $\mathbb{P}_\delta \Vdash \dot{\mathbb{Q}}_\delta = S^{W_\delta}$, so $W_{\delta+1}$ contains a filter $G_0 \subseteq S^{W_\delta}$ intersecting all $A_\gamma$'s (note that by absoluteness of $\psi^{\in}$ and $\psi^{\bot}$, the sets $A_\gamma$ are maximal antichains in $S^{W^\delta}$). The filter generated by $G_0$ in $S^{W_\kappa}$ is the required generic filter.$\qedhere$
\end{proof}

\par If $N$ is a transitive class containing $\kappa$, we can define by induction the relativized iteration $\mathbb{P}^N_\kappa\subseteq \mathbb{P}_\kappa$, taking into account only names from $N$. 

\begin{itemize}
	\item $\mathbb{P}^N_0=\{0\}$,
	\item $\mathbb{P}^N_\alpha \Vdash``\dot{\mathbb{Q}}^N_\alpha=\{x \in \baire|\;\psi^\in(x,\dot{r}_\alpha) \} \text{ if this formula defines a Suslin forcing,}\\
		 \text{ $\dot{r}_\alpha \in N$, and $\dot{r}_\alpha$ is a $\mathbb{P}_\alpha^N$-name; else } \dot{\mathbb{Q}}_\alpha^N=\{0\}"$,
	\item $\mathbb{P}^N_{\alpha+1}=\mathbb{P}^N_\alpha\ast \dot{\mathbb{Q}}^N_\alpha$.
	
\end{itemize}

We take direct limits in the limit step, so that $\mathbb{P}_\alpha^N$ is really a subset of $\mathbb{P}_\alpha$. Note, that we do not define names $\dot{r}_\alpha$ inductively along the way, since they have already been defined in the construction of $\mathbb{P}_\kappa$, which we take as granted. This construction is inspired by the lemmas 1.4 and 1.5 from \cite{292}, and conceptually is very similar. In order for it to work as desired, we prove by induction some properties of $\mathbb{P}_\alpha^N$.

\begin{lem} \label{iterationlemma}
	If $G\subseteq \mathbb{P}_\alpha$ is generic over $V$, and $H\subseteq \dot{\mathbb{Q}}_\alpha[G]$ is generic over $V[G]$, then $H\cap \dot{\mathbb{Q}}^N_\alpha[G]\subseteq \dot{\mathbb{Q}}^N_\alpha[G]$ is generic over $V[G\cap \mathbb{P}_\alpha^N]$.
\end{lem}
\begin{proof}
	Fix a maximal antichain $\mathcal{A} \subseteq \dot{\mathbb{Q}}^N_\alpha[G]=\dot{\mathbb{Q}}^N_\alpha[G']$, belonging to $V[G']$. As $\mathcal{A}$ is a countable set of reals, it can be coded using a single real $z \in \baire$. Recall that $\dot{\mathbb{Q}}^N_\alpha[G']$ is defined in $V[G']$ by the formula $\psi$ with the parameter $\dot{r}_\alpha[G']=\dot{r}_\alpha[G]$ (unless it is a trivial forcing). It is standard to check, that there exists a $\coanal$ formula $\phi(-,-)$ such that whenever $\psi(-,-,-,y)$ defines a Suslin forcing, 
	\begin{center}
		$\phi(x,y) \, \iff\, ``x$ is a real coding a maximal antichain in the partial ordering defined by the formula $\psi(-,-,-,y)"$.
	\end{center}
	Now
	$$V[G']\models \phi(z,\dot{r}_\alpha[G']),$$
	and so by absoluteness
	$$V[G]\models \phi(z,\dot{r}_\alpha[G]).$$
	But $\psi(-,\dot{r}_\alpha[G])$ is the formula defining $\dot{\mathbb{Q}}_\alpha[G]$ in $V[G]$. Therefore $\mathcal{A}$ remains maximal in $\dot{\mathbb{Q}}_\alpha[G]$, and the conclusion of the Lemma easily follows.$\qedhere$
\end{proof}

\begin{thm}\label{bigthm} If $N$ is a transitive class containing $\kappa$, then for all $\alpha\le \kappa$ $$\mathbb{P}^N_\alpha \lessdot \mathbb{P}_\alpha.$$ Specifically:
	
\begin{enumerate}
	\item If $p_0 \bot p_1$ in $\mathbb{P}_\alpha^N$, then $p_0 \bot p_1$ in $\mathbb{P}_\alpha$.
	\item If $p_0 \le p_1$ in $\mathbb{P}_\alpha^N$, then $p_0 \le p_1$ in $\mathbb{P}_\alpha$.
	\item If $\mathcal{A}\subseteq \mathbb{P}^N_\alpha$ is a maximal antichain, then $\mathcal{A}$ is maximal in $\mathbb{P}_\alpha$.
\end{enumerate}
\end{thm}

\begin{proof}$\text{We proceed by induction on } \alpha.$
	\par{Concerning the point 1.}
	\begin{itemize}
		\item $\alpha=0$. Clear.
		\item $\alpha+1$. Fix two incompatible conditions $p_0,p_1 \in \mathbb{P}_{\alpha+1}^N$. Then $p_0=(p'_0,\dot{q}_0)$, and $p_1=(p'_1,\dot{q}_1)$, where $p'_0, p'_1 \in \mathbb{P}_\alpha^N$. If $p'_0 \bot p'_1$ in $\mathbb{P}_\alpha$, then clearly $p_0 \bot p_1$ in $\mathbb{P}_{\alpha+1}$, so let us assume otherwise, and fix $p\le p'_0,p'_1$ (in $\mathbb{P}_\alpha$). If there exists $p^* \le p$ such that $p^* \Vdash \dot{\mathbb{Q}}^N_\alpha=\{0\}$, then $(p^*,0)\le p_0,p_1$, contradicting $p_0 \bot p_1$. Therefore $p$ forces that $\dot{\mathbb{Q}}^N_\alpha$ is defined by $\psi^\in(-,\dot{r}_\alpha)$. The forcing relation used above is a relation from $\mathbb{P}^N_\alpha$, however since $\dot{r}_\alpha$, $\dot{q}_0$ and $\dot{q}_1$ are $\mathbb{P}^N_\alpha$-names, this is the same relation as coming from $\mathbb{P}_\alpha$ (see Proposition \ref{aswellas}). We aim to show that $p_0 \bot p_1$ in $\mathbb{P}_{\alpha+1}$.
			
		 Let $p \in G\subseteq \mathbb{P}_\alpha$ be a filter generic over $V$. Conditions $p_0$ and $p_1$ were incompatible in $\mathbb{P}^N_{\alpha+1}$ and, by the induction hypothesis, $G\cap \mathbb{P}^N_\alpha \subseteq \mathbb{P}^N_\alpha$ is generic over $V$, therefore
		$$V[G\cap \mathbb{P}^N_\alpha] \models \psi^{\bot}(\dot{q}_0[G],\dot{q}_1[G],\dot{r}_{\alpha}[G]).$$
		
		By absoluteness
		
		$$V[G] \models \psi^{\bot}(\dot{q}_0[G],\dot{q}_1[G],\dot{r}_{\alpha}[G]).$$
		
		Since $p$ was arbitrary, it follows that $p_0 \bot p_1$ in $\mathbb{P}_{\alpha+1}$.
		\item $\lim{\alpha}$. Follows from the induction hypothesis, since conditions have finite supports.
	\end{itemize}
	
	\par{Concerning the point 2.} 	
	\begin{itemize}
		\item $\alpha=0$. Clear.
		\item $\alpha+1$. Fix two conditions $p_0\le p_1 \in \mathbb{P}_{\alpha+1}^N$. Then $p_0=(p'_0,\dot{q}_0)$, $p_1=(p'_1,\dot{q}_1)$, where $p'_0, p'_1 \in \mathbb{P}_\alpha^N$. By the induction hypothesis $p_0'\le p_1'$ in $\mathbb{P}_\alpha$. Moreover $\dot{r}_\alpha$, $\dot{q}_0$ and $\dot{q}_1$ are $\mathbb{P}^N_\alpha$-names, so -- in the light of Proposition \ref{aswellas} -- the forcing relation
		$$p_0' \Vdash \dot{q}_0 \le \dot{q}_1$$
		holds in $\mathbb{P}^N_\alpha$ as well as in $\mathbb{P}_\alpha$.
		
		\item $\lim{\alpha}$. Follows from the induction hypothesis, since conditions have finite supports.
	\end{itemize}

	\par{Concerning the point 3.} 
		\begin{itemize}
			\item $\alpha=0$. Clear.
			
			\item $\lim{\alpha}$. Let $\mathcal{A}$ be a maximal antichain in $\mathbb{P}^N_\alpha$, and $\overline{p} \in \mathcal{A}$. Given that $\mathbb{P}_\alpha^N$ is a finite support iteration of c.c.c. forcings, it satisfies the countable chains condition, therefore we can assume that $\mathcal{A}=\{\overline{p}_n|\; n<\omega\}$. Since $\alpha$ is limit, there is some $\gamma < \alpha$ such that $\overline{p} \in \mathbb{P}_\gamma$. $\{\overline{p}_n\restriction \gamma|\; n<\omega\}$ might not be an antichain in $\mathbb{P}^N_\gamma$, however each condition in $\mathbb{P}^N_\gamma$ is compatible with some $p_n\restriction \gamma$. We can refine $\{\overline{p}_n\restriction \gamma|\; n<\omega\}$ to an antichain in $\mathbb{P}^N_\gamma$, and this antichain will remain maximal
			in $\mathbb{P}_\gamma$ by the induction hypothesis. Therefore $\{\overline{p}_n\restriction \gamma|\; n<\omega\}$ meets every condition in $\mathbb{P}_\gamma$, and in particular some $\overline{p}_n \restriction \gamma$ is compatible with $\overline{p}$ in $\mathbb{P}_\gamma$. But then $\overline{p}_n$ is compatible with $\overline{p}$ in $\mathbb{P}_\alpha$. 
			
			\item $\alpha+1$. In the light of Proposition \ref{observation}, it is sufficient to show that for any $G\subseteq \mathbb{P}_{\alpha+1}$ generic over $V$, $G \cap \mathbb{P}^N_{\alpha+1}$ is also generic over $V$. We will shows how it follows from Lemma \ref{iterationlemma}. Let $\overline{G}\subseteq \mathbb{P}_\alpha \ast \dot{\mathbb{Q}}_\alpha$ be a filter generic over $V$. Recalling the notation from \cite{kunen}, 
			$$\overline{G}=G\ast H=\{(p,\dot{q})|\;p \in G,\; \dot{q}[G]\in  H  \},$$
			
			where 
			$$G=\{p \in \mathbb{P}_\alpha|\; \exists \dot{q} \;\in \dot{\mathbb{Q}} \quad (p,\dot{q}) \in \overline{G} \},$$
			and
			$$H=\{\dot{q}[G]|\; \exists p \in G \quad (p,\dot{q})\; \in \overline{G}\}.$$
			
			It is known that for any iteration $\mathbb{P}\ast \dot{\mathbb{Q}}$, if $G\subseteq \mathbb{P}$ is generic over $V$ and $H\subseteq \dot{\mathbb{Q}}[G]$ is generic over $V[G]$, then $G\ast H$ is generic for $\mathbb{P}\ast\dot{\mathbb{Q}}$ over $V$ (for details consult for example \cite{kunen}, Section 5, Chapter VIII).	Let $G'=G \cap \mathbb{P}^N_\alpha$. It is generic for $\mathbb{P}^N_\alpha$ over $V$ by the induction hypothesis. Now for filters $G$ and $H$ defined above
			
			$$(G\ast H) \cap (\mathbb{P}^N_\alpha \ast \dot{\mathbb{Q}}^N_\alpha)=
			\{ (p,\dot{q})|\; p \in G',\; \dot{q}[G] \in H,\; \dot{q} \in \dot{\mathbb{Q}}^N_\alpha \}=$$
			$$\{ (p,\dot{q})\in \mathbb{P}^N_\alpha \ast \dot{\mathbb{Q}}^N_\alpha |\; p \in G',\; \dot{q}[G'] \in H\}=G'\ast(H\cap \dot{\mathbb{Q}}^N_\alpha[G']).$$
			
			But from the Lemma \ref{iterationlemma} we know that this is a $\mathbb{P}^N_\alpha \ast \dot{\mathbb{Q}}^N_\alpha$-generic filter over $V$. 
			
This concludes the proof of Theorem \ref{bigthm}.$\qedhere$
		\end{itemize}	
\end{proof}

Even if $N$ is an inner model of $\zfc$, usually $\mathbb{P}^N_\kappa \notin N$. Definition of $\mathbb{P}^N_\kappa$ makes use of a list of $\mathbb{P}^N_\alpha$-names, for all $\alpha<\kappa$, and although \emph{some} such enumeration belongs to $N$ (as it is a model of $\ac$), this particular might not. In what sense is $\mathbb{P}^N_\kappa$ a \emph{relativized} version of $\mathbb{P}_\kappa$, is explained by the next lemma.

\begin{lem} \label{definable}
	For each $\alpha\le \kappa$, if $p \in \mathbb{P}_\alpha$ is simple then $p$ is definable (in the language of set theory) with a parameter from $\kappa^\omega$.
\end{lem}

\begin{proof}$\text{We proceed by induction on }\alpha.$
	\begin{itemize}
		\item $\alpha=0.$ Clear, since $\mathbb{P}_0$ is a trivial forcing.
		\item $\alpha+1$. Let $r=(p,\dot{q})$ be simple. We can write 
		$$\dot{q}=\{(m,n,p^m_n)|\; m,n \in \omega,\, p^m_n \in \mathbb{P}_\alpha\},$$
		 where each $p^m_n$ is simple. By the induction hypothesis each $p^m_n$ is definable with a parameter from $\kappa^\omega$, and so is $p$. Clearly $r$ can be defined from them, and so $r$ is definable with countably many parameters from $\kappa^\omega$. We can easily code them as a single parameter.
		\item $\lim{\alpha}$. Fix $p \in \mathbb{P}_\alpha$. $p$ has finite support, so there exists $\beta<\alpha$ containing the support of $p$. By the induction hypothesis $p\restriction \beta$ is definable with a parameter from $\kappa^\omega$, and $p$ is definable with parameters $p\restriction \beta$, $\beta$, and $\alpha$.$\qedhere$
	\end{itemize}
\end{proof}

From this point, we fix a list $\{\sigma_\alpha|\;\alpha<\kappa \}$ of sequences from $\kappa^\omega$, such that each name $\dot{r}_\alpha$ is definable from $\sigma_\alpha$.

\begin{lem} \label{inclusionlemma}
	Let $N$ be any transitive model of $\zfc$ containing $\kappa$, and let $M\prec H((2^\kappa)^+)$ be a countable elementary submodel such that $\{\sigma_\alpha|\;\alpha<\kappa \}\in M$, and $M\cap \kappa^\omega \subseteq N$. Then for any $\alpha\le \kappa$, $\mathbb{P}_\alpha \cap M \subseteq \mathbb{P}_\alpha^N$.
\end{lem}

\begin{proof}
	We proceed by induction. 
	
	\begin{itemize}
		\item $\alpha=0$. Clear.
		\item $\lim{\alpha}.$ Fix $r \in \mathbb{P}_\alpha \cap M$. By the elementarity of $M$, there exists $\gamma \in \alpha\cap M$ that contains the support of $r$. From the induction hypothesis it follows that $r \restriction \gamma \in \mathbb{P}_\gamma^N$. It is routine to verify by induction that for all $\gamma\le \delta\le \alpha$, $r\restriction \delta \in \mathbb{P}_\delta^N$.
		
		\item $\alpha+1.$ Fix $r=(p,\dot{q}) \in M \cap (\mathbb{P}_\alpha \ast \dot{\mathbb{Q}}_\alpha)$. We know that $p \in \mathbb{P}_\alpha^N$ by the induction hypothesis. The condition $\dot{q}$ is of the form $\dot{q}=\{(m,n,p^m_n)|\;m,n<\omega, \; p^m_n \in \mathbb{P}_\alpha  \}$. Given that all conditions $p^m_n$ belong to $M$, they also belong to $\mathbb{P}_\alpha^N$ by the induction hypothesis. This shows that $\dot{q}$ is a $\mathbb{P}^N_\alpha$-name. To see that $p \Vdash \dot{q} \in \dot{\mathbb{Q}}_\alpha^N$, fix a generic filter $G\subseteq \mathbb{P}_\alpha$ containing $p$ (by Theorem \ref{bigthm} and Lemma \ref{aswellas} it is irrelevant whether we consider the $\Vdash$ relation in $\mathbb{P}_\alpha$ or $\mathbb{P}_\alpha^N$). We have two cases to consider.
		\begin{itemize}
			\item $\dot{\mathbb{Q}}_\alpha[G]=\{0\}$.
			
			First, $\dot{r}_\alpha$ is a $\mathbb{P}^N_\alpha$-name for the very same reason as $\dot{q}$. Second, given that $\sigma_\alpha \in \kappa^\omega \cap M \subseteq N$, we know that $\dot{r}_\alpha \in N$. It follows that this case will occur only if $\psi^\in(-,\dot{r}_\alpha[G])$ does not define a Suslin forcing. But by absoluteness, it also doesn't define a Suslin forcing in $V[G\cap \mathbb{P}^N_\alpha]$, and so $\dot{q}[G]=0\in\dot{\mathbb{Q}}^N_\alpha[G]$.
			
			\item $\dot{\mathbb{Q}}_\alpha[G]$ is defined by $\psi^\in(-,\dot{r}_\alpha[G])$.
			
			As in the previous case, we know that $\dot{r}_\alpha$ is a $\mathbb{P}_\alpha^N$-name belonging to $N$, and  so $\psi^\in(\dot{q}[G],\dot{r}_\alpha[G])$ holds in $V[G\cap \mathbb{P}_\alpha^N]$. But this just says that $\dot{q}[G]\in \dot{\mathbb{Q}}^N_\alpha[G]$.$\qedhere$
		\end{itemize}
		
		
	\end{itemize}
\end{proof}

\begin{proof}[Proof of Theorem \ref{mainthm}]
	We start with a model $V\models \slice + ``2^\omega=\kappa"$, and we assume moreover that the sequence $\{M_\alpha|\; \alpha<\omega_1\}$ witnessing $\slice$ satisfies the following stronger property:
	
	$$\kappa^{\omega}= \displaystyle{ \bigcup_{\alpha <\omega_1} \kappa^{\omega} \cap M_\alpha}.$$
	
	Such a model is easy to get, for example by forcing with $\mathbb{C}_\kappa$ over a model of $\ch$ (see Corollary \ref{cohen}). We also assume that $\kappa \in M_0$.
	
	\par Let $\mathbb{P}=\{\mathbb{P}_\alpha\ast \dot{\mathbb{Q}}_\alpha|\; \alpha<\kappa\}$ be the iteration described above, which forces 
	$$\ma(\text{Suslin})+``2^\omega=\kappa".$$
	We claim that if $G\subseteq \mathbb{P}$ is generic over $V$, then the sequence $V[G\cap\mathbb{P}^{M_\alpha}]$ witnesses $\slice$ in $V[G]$. For this we need to show two things
	
	\begin{enumerate}
		\item If $r \in \omega^\omega\cap V[G]$, then $r \in V[G\cap\mathbb{P}^{M_\alpha}]$ for some $\alpha<\omega_1$.
		\item None of the models $V[G\cap\mathbb{P}^{M_\alpha}]$ contains all reals from $V[G]$\footnote{So, strictly speaking, $\slice$ won't be witnessed by $\{M_\alpha|\; \alpha<\omega_1\}$, rather some its subsequence}.
	\end{enumerate}
	
	\par Concerning $1.$ suppose that $\mathbb{P}_\kappa \Vdash \dot{r}\in \omega^\omega$. We can assume that 	$$\dot{r}=\{(m,n,p_n^m)|\; m,n<\omega \},$$
	and all conditions $p_n^m$ are simple. Fix a countable elementary submodel $\overline{M}\prec \operatorname{H}((2^\kappa)^+)$, that contains the list $\{\sigma_\alpha|\; \alpha<\kappa\}$, and the name $\dot{r}$. We pick $\alpha$ big enough so that $\overline{M}\cap \kappa^\omega\subseteq M_\alpha$. Applying Lemma \ref{inclusionlemma} with $M=\overline{M}$ and $N=M_\alpha$, we see that $\dot{r}$ is a $\mathbb{P}^{M_\alpha}$-name. Therefore 
	$$\dot{r}[G]=\dot{r}[G\cap \mathbb{P}^{M_\alpha}] \in V[G\cap\mathbb{P}^{M_\alpha}].$$
	
	\par Concerning $2.$ fix a real $r \in \baire \setminus M_\alpha$. We can find a representation of the Cohen forcing $\mathbb{C}$, such that the real $r$ is definable from it\footnote{The same is obviously true for any other Suslin forcing. We choose $\mathbb{C}$ for simplicity. The only point of this, is to argue that $\mathbb{P}^{M_\alpha}$ will vanish on some coordinate, on which $\mathbb{P}$ will add a real.}. For concreteness, let us put
	
	$$\mathbb{C}_r=\omega^{<\omega}\cup \{r\} \subseteq \baire,$$
	where $\omega^{<\omega}$ is identified with the set of sequences from $\baire$ that are eventually equal zero. We order $\omega^{<\omega}$ by the end-extension and we declare that
	$$\forall{s \in \omega^{<\omega}}\; s \bot r.$$
	
	Since $\mathbb{C}_r$ is clearly Suslin, there exists a real $r'$ such that
	$$\mathbb{C}_r = \{x \in \baire|\; \psi^{\in}(x,r') \}.$$
	
	We claim that $r' \notin M_\alpha$. Suppose otherwise. Let $\sigma(x)$ stand for the formula
	$$\psi^{\in}(x,r') \wedge x \notin \omega^{<\omega}.$$
	Note that
	$$V \models \exists\; x \in \baire \; \sigma(x),$$
	and so by absoluteness the same holds in $M_\alpha$. Fix $r'' \in \baire \cap M_\alpha$, such that
	$$M_\alpha \models \sigma(r'').$$
	Again $V \models \sigma(r'')$ by absoluteness. But this shows that $r=r''$, and therefore $r \in M_\alpha$, contradicting the choice of $r$.
	
	Once we know that $r' \notin M_\alpha$, let us fix $\gamma <\kappa$ such that $\mathbb{P}_\gamma \Vdash \dot{r}_\gamma = r'$. It follows that
	$$\mathbb{P}^{M_\alpha}_\gamma \Vdash \dot{\mathbb{Q}}^{M_\alpha}_\gamma=\{0\},$$
	and
	$$\mathbb{P}_\gamma \Vdash \dot{\mathbb{Q}}_\gamma=\mathbb{C}_r.$$
	
	Let $c$ be the real added by $\dot{\mathbb{Q}}_\gamma$ over $V[G\cap \mathbb{P}_\gamma]$. We claim that $c \notin V[G\cap \mathbb{P}^{M_\alpha}_\kappa]$. Suppose towards a contradiction that $c \in V[G\cap \mathbb{P}^{M_\alpha}_\kappa]$. Let $\mathbb{R}=\mathbb{P}_\kappa^{M_\alpha}/(G\cap \mathbb{P}^{M_\alpha}_{\gamma+1})$. If $c \in V[G\cap \mathbb{P}^{M_\alpha}_\kappa]$, then either $c$ is already in the initial fragment $V[G\cap \mathbb{P}^{M_\alpha}_{\gamma+1}]$, or is added generically over it by $\mathbb{R}$. The first option is ruled out, because 
	$$V[G\cap \mathbb{P}^{M_\alpha}_{\gamma+1}]=V[G\cap \mathbb{P}^{M_\alpha}_\gamma]\subseteq V[G\cap \mathbb{P}_\gamma],$$
	and $c$ was added generically over that model. Therefore in $V[G\cap \mathbb{P}^{M_\alpha}_{\gamma+1}]$ there is a $\mathbb{R}$-name $\dot{x}$ such that two of the following hold:
	\begin{itemize}
		\item $\mathbb{R} \Vdash ``\text{$\dot{x}$ is a new real}"$,
		\item $\dot{x}[G\cap \mathbb{P}_\kappa^{M_\alpha}]=c.$
	\end{itemize}
	
	Consider now the forcing $\mathbb{R}$ in $V[G\cap \mathbb{P}_{\gamma+1}]$. It is still true that $\mathbb{R}$ forces $\dot{x}$ to be a new real, since this is an absolute property of the forcing and the name (it says that for any $r \in \mathbb{R}$ there is $n<\omega$, such that $r$ does not decide $\dot{x}(n)$). But this is a contradiction, since $c=\dot{x}[G] \in V[G\cap \mathbb{P}_{\gamma+1}]$.
	
This concludes the proof of Theorem \ref{mainthm}.$\qedhere$
\end{proof}
\subsection{Computation of cardinal invariants}

For any given group $G$, one can study an associated cardinal invariant $\operatorname{c(G)}$ that stands for the minimal cardinality $\kappa$, for which the group $G$ can be represented as a union of a chain of the length $\kappa$, consisting of proper subgroups of $G$. A substantial amount of literature is devoted to study this cardinal invariant for symmetric groups of infinite sets (for example \cite{c1}, \cite{c2}, \cite{c3}, \cite{c4}). It is known that 
$$\mathfrak{g}\le\sym \le \mathfrak{d},$$
where $\mathfrak{g}$ is the \emph{groupwise density number}. The lower bound was proved by Brendle and Losada \cite{bl}, and the upper bound is due to Sharp and Thomas \cite{st}.
\par It is easy to observe that $\slice \implies ``\sym=\omega_1"$: if $\slice$ is witnessed by a sequence $(M_\alpha)_{\alpha<\omega_1}$, then the equality $\sym=\omega_1$ is witnessed by the sequence of groups $(M_\alpha \cap \operatorname{Sym}(\omega))_{\alpha<\omega_1}$. As a matter of fact, this observation shows that Theorem \ref{mainthm} generalizes Lemma 2.6 from \cite{zhang}, which claims that the equality $\sym=\omega_1$ is preserved under finite support iterations of Suslin forcings. Indeed, what the proof of Theorem \ref{mainthm} shows, is that any finite support iteration of Suslin forcings preserves $\slice$ over a model of a stronger variant of $\slice$.

Together with some well-known results (see \cite{blass}), we have the following series of inequalities:
$$\sym\ge\mathfrak{g}\ge\mathfrak{h}\ge \mathfrak{t}\ge \mathfrak{m} \ge \omega_1.$$
It follows that in our model all these invariants are equal $\omega_1$. Together with the fact that $\ma(\text{Suslin})$ decides all cardinal characteristics from the Cicho\'n's diagram to be equal $2^\omega$, we have computed all of the classical cardinal invariants of the continuum, except $\mathfrak{s}$.

\section{Slicing $2^{<\kappa}$}

Although $\ma_{\omega_1}$ is inconsistent with $\slice$, it is consistent with $\slicekappa$ for any $\kappa>\omega_1$. The idea of the proof is very much like that of Theorem \ref{mainthm}, and actually even simpler, because we don't need to code the steps of the iteration as analytic sets.

\begin{thm} \label{mainthm2}
Assume that $\omega<\kappa\le\theta$ are regular cardinals, and $\theta^{<\kappa}=\theta$. Then the following theory is consistent for any cardinal $\lambda<\kappa$:
$$\zfc+\ma_{\lambda}+\slicekappa + ``2^\omega=\theta".$$
\end{thm}

We are going to apply a finite support iteration of the form
$$\mathbb{P}=\{\mathbb{P}_\alpha \ast \dot{\mathbb{Q}}_\alpha|\; \alpha<\theta \},$$
where for each $\alpha<\theta$
$$\mathbb{P}_\alpha \Vdash \dot{\mathbb{Q}}_\alpha=(\lambda,\dot{\le}_\alpha).$$
We also assume that $0\in \lambda$ is always the largest element in $\dot{\mathbb{Q}}_\alpha$. We want to arrange the iteration so that each c.c.c. partial order of size $\lambda$ will appear cofinally many times (see \cite{kunen}, p. 278), and for this reason, we will be considering only names of the form

$$\dot{\le}_\alpha =\{ (\phi(\beta),p_i^\beta)|\; i<\omega, \beta<\lambda \},$$
where $\phi:\lambda \rightarrow \lambda\times\lambda$ is a fixed bijection, definable from $\lambda$. A standard induction shows that for any $\alpha\le \theta$ there exists at most $\theta$-many such names, and $|\mathbb{P}_\alpha|\le \theta$. Using an appropriate bookkeeping we can include all c.c.c. partial orders of size $\lambda$ in our iteration, and therefore we obtain:

\begin{thm} Under the assumptions of Theorem \ref{mainthm2}
	$$\mathbb{P}_\theta \Vdash \ma_\lambda + ``2^\omega=\theta".$$
\end{thm}

\begin{defin} By induction on $\alpha$, we define the class of \emph{simple} $\mathbb{P}_\alpha$-conditions.
	\begin{itemize}
		\item $\alpha=0.$ $\mathbb{P}_0=\{0\}$, and we declare $0$ to be simple.
		\item $\alpha+1.$ $(p,\dot{q})\in \mathbb{P}_{\alpha+1}$ is simple if $p \in \mathbb{P}_\alpha$ is simple, $\dot{q}=\{(\gamma_n,p_n)|\; n<\omega\}$, and the conditions $p_n$ are simple.
		\item $\lim{\alpha}.$ $p \in \mathbb{P}_\alpha$ is simple if for each $\beta<\alpha$, $p \restriction \beta \in \mathbb{P}_\beta$ is simple.
	\end{itemize}
\end{defin}
Like in the previous section, it is easy to check that the set of simple conditions is always dense.

\begin{lem} \label{definable2}
	For each $\alpha\le \theta$, if $p \in \mathbb{P}_\alpha$ is simple then $p$ is definable (in the language of set theory) with a parameter from $\theta^\omega$.
\end{lem}

\begin{proof} $\text{  }$
	\begin{itemize}
		\item $\alpha=0.$ Clear.
		\item $\alpha+1$. Let $r=(p,\dot{q})$ be a simple condition. We can write $\dot{q}=\{(\gamma_n,p_n)|\;n<\omega \}$, where conditions $p_n$ are simple. By the induction hypothesis each $p_n$ is definable with a parameter from $\theta^\omega$, and so is $p$. Clearly $r$ can be defined from them, and so $r$ is definable with countably many parameters, which we can code as one.	
		\item $\lim{\alpha}$. Fix $r \in \mathbb{P}_\alpha$. $r$ has finite support, so there exists $\beta<\alpha$ containing the support of $r$. By the induction hypothesis $p\restriction \beta$ is definable with a parameter from $\theta^\omega$, and so $p$ is definable with the parameters $p\restriction \beta$, $\beta$, and $\alpha$.$\qedhere$
	\end{itemize}
\end{proof}

An immediate consequence is that each of the names $\dot{\le}_\alpha$ is definable with some parameter $\sigma_\alpha \in \theta^\lambda$. Like previously, we fix a list of such parameters $\{\sigma_\alpha|\; \alpha<\theta\}\subseteq \theta^\lambda$. We define by induction the relativized forcings $\mathbb{P}_\kappa^N \subseteq \mathbb{P}_\kappa$, taking into account only names from some transitive class $N$.

\begin{itemize}
	\item $\mathbb{P}^N_0=\{0\}$,
	\item Assume $\mathbb{P}^N_\alpha$ is defined. We define a $\mathbb{P}^N_\alpha$-name $\dot{\mathbb{Q}}^N_\alpha$ as follows
	
		\begin{itemize}
			\item $\dot{\mathbb{Q}}^N_\alpha=\dot{\mathbb{Q}}_\alpha$ if $\dot{\mathbb{Q}}_\alpha \in N$, and 
			$\dot{\mathbb{Q}}_\alpha$ is a $\mathbb{P}^N_\alpha$-name,
			\item $\dot{\mathbb{Q}}^N_\alpha=\{0\}$ otherwise.
		\end{itemize}
	\item $\mathbb{P}^N_{\alpha+1}=\mathbb{P}^N_\alpha\ast \dot{\mathbb{Q}}^N_\alpha$.
\end{itemize}

In limit steps we take direct limits, so $\mathbb{P}^N_\kappa\subseteq \mathbb{P}_\kappa$. 

\begin{lem}
	Let $N$ be a transitive model of $\zfc$, containing $\theta$. Let $M\prec H((2^\theta)^+)$ be an elementary submodel, such that $\lambda+1 \subseteq M$, and $\{\sigma_\alpha|\; \alpha <\theta \}\in M$ (see the remark after Lemma \ref{definable2}). We assume moreover that $\theta^\lambda \cap M \subseteq N$. Then for each $\alpha\le \theta$, $\mathbb{P}_\alpha \cap M \subseteq \mathbb{P}_\alpha^N$.
\end{lem}

\begin{proof}
	We proceed by induction. 
	
	\begin{itemize}
		\item $\alpha=0$. Clear.
		\item $\lim{\alpha}.$ Fix $r \in \mathbb{P}_\alpha \cap M$. By the elementarity of $M$, there exists $\gamma \in \alpha\cap M$ that contains the support of $r$. From the induction hypothesis it follows that $r \restriction \gamma \in \mathbb{P}_\gamma^N$. It is routine to verify by induction that for all $\gamma\le \delta\le \alpha$, $r\restriction \delta \in \mathbb{P}_\delta^N$.
		\item $\alpha+1.$ Fix $r=(p,\dot{q}) \in M \cap (\mathbb{P}_\alpha \ast \dot{\mathbb{Q}}_\alpha)$. Clearly $p \in \mathbb{P}_\alpha^N$ by the induction hypothesis. The name $\dot{q}$ is of the form
		$$\dot{q}=\{(\gamma_n,p_n)|\; n<\omega\},$$
		and for each $n<\omega$, $p_n \in \mathbb{P}_\alpha \cap M \subseteq \mathbb{P}_\alpha^N$. This shows that $\dot{q}$ is a $\mathbb{P}^N_\alpha$-name. it remains to show that $\mathbb{P}_\alpha^N\Vdash \dot{\mathbb{Q}}_\alpha^N=\dot{\mathbb{Q}}_\alpha$, and this in turn reduces to showing that $\dot{\mathbb{Q}}_\alpha$ is a $\mathbb{P}^N_\alpha$-name belonging to $N$. To see this, let us note that, since $(p,\dot{q}) \in M$, also $\alpha \in M$, and so $\sigma_\alpha \in \theta^\lambda \cap M\subseteq N$. It follows that $\dot{\mathbb{Q}}_\alpha \in N\cap M$. Recall, that $\dot{\mathbb{Q}}_\alpha$ is a $\mathbb{P}_\alpha$-name for a partial ordering of the form
		$$\dot{\le}_\alpha=\{ (\phi(\beta),p^i_\beta )|\;\beta<\lambda, i<\omega \}.$$
		Given that $\lambda+1 \subseteq M$, we conclude that each of the conditions $p^i_\beta$ belongs to $M$, and  by the induction hypothesis, also to $\mathbb{P}_\alpha^N$. This shows that $\dot{\le}_\alpha$, and in turn also $\dot{\mathbb{Q}}_\alpha$, are $\mathbb{P}_\alpha^N$-names, and concludes the proof.$\qedhere$
	\end{itemize}
\end{proof}

\begin{lem} \label{generalconmpletesuborders}
	If $N$ is a transitive class, then for all $\alpha \le \theta$
	$$\mathbb{P}^N_\alpha \lessdot\mathbb{P}_\alpha.$$
	Specifically:
	\begin{enumerate}
		\item If $p_0 \bot p_1$ in $\mathbb{P}_\alpha^N$, then $p_0 \bot p_1$ in $\mathbb{P}_\alpha$.
		\item If $p_0 \le p_1$ in $\mathbb{P}_\alpha^N$, then $p_0 \le p_1$ in $\mathbb{P}_\alpha$.
		\item If $\mathcal{A}\subseteq \mathbb{P}^N_\alpha$ is a maximal antichain, then $\mathcal{A}$ is maximal in $\mathbb{P}_\alpha$.
	\end{enumerate}	
\end{lem}

\begin{proof}
	We proceed by induction on $\alpha$. 
	\par{Concerning the point 1.}
	\begin{itemize}
		\item $\alpha=0.$ Clear.
		\item $\alpha+1.$ Assume $(p_0,\dot{q}_0)\bot(p_1,\dot{q}_1)$ in $\mathbb{P}^N_{\alpha+1}$. If $p_0\bot p_1$ in $\mathbb{P}_\alpha^N$, then by the induction hypothesis $p_0\bot p_1$ in $\mathbb{P}_\alpha$ and we are done. Suppose otherwise, and fix a condition $p\le p_0,p_1$ from $\mathbb{P}_\alpha$. If $p$ can be extended to a $p^*$ such that $p^* \Vdash \dot{\mathbb{Q}}_\alpha^N=\{0\}$, then $p^* \le (p_0,\dot{q}_0),(p_1,\dot{q}_1)$, contradicting $(p_0,\dot{q}_0)\bot(p_1,\dot{q}_1)$. So $p$ forces that $\dot{\mathbb{Q}}_\alpha^N$ is defined in the nontrivial way. Let $G\subseteq \mathbb{P}_\alpha$ be any filter generic over $V$, containing $p$. Since $p_0,p_1 \in G \cap \mathbb{P}_\alpha^N$, we see that 
		$$\dot{q}_0[G\cap \mathbb{P}_\alpha^N] \bot \dot{q}_1[G\cap \mathbb{P}_\alpha^N]$$
		in the model $V[G\cap \mathbb{P}_\alpha^N]$, and so in $V[G]$ as well (see Proposition \ref{aswellas}). Since $p$ and $G$ were arbitrary, it follows that $(p_0,\dot{q}_0)\bot(p_1,\dot{q}_1)$ in $\mathbb{P}_{\alpha+1}$.
		\item $\lim{\alpha}.$ Follows from the induction hypothesis, since the supports are finite. 
	\end{itemize}
	\par{Concerning the point 2.}
	\begin{itemize}
		\item $\alpha=0.$ Clear.
		\item $\alpha+1.$ Assume $(p_0,\dot{q}_0)\le(p_1,\dot{q}_1)$ in $\mathbb{P}^N_{\alpha+1}$. From the induction hypothesis we know, that $p_0\le p_1$ in $\mathbb{P}_\alpha$, and $p_0 \Vdash \dot{q}_0 \le \dot{q}_1$ in $\mathbb{P}_\alpha^N$. We must show that the assertion
		$$p_0 \Vdash \dot{q}_0 \le \dot{q}_1$$
		holds also in $\mathbb{P}_\alpha$. But it follows outright from Proposition \ref{aswellas}.
		
		\item $\lim{\alpha}.$ Follows from the induction hypothesis, since the supports are finite. 
	\end{itemize}
	\par{Concerning the point 3.}
	\begin{itemize}
		\item $\alpha=0.$ Clear.
		\item $\alpha+1.$ The proof is exactly the same, as in the successor case of Theorem \ref{bigthm} point 3, except this time the conclusion of Lemma \ref{iterationlemma} follows straightforward from the relation
		$$\mathbb{P}^N_\alpha \Vdash \dot{\mathbb{Q}}^N_\alpha \lessdot \dot{\mathbb{Q}}_\alpha.$$		
		\item $\lim{\alpha}$. Let $\{\overline{p}_n|\; n<\omega\}$ be a maximal antichain in $\mathbb{P}^N_\alpha$, and fix $\overline{p} \in \mathbb{P}_\alpha$. There is some $\gamma < \alpha$ such that $\overline{p} \in \mathbb{P}_\gamma$. The set $\{\overline{p}_n\restriction \gamma|\; n<\omega\}$ might not be an antichain in $\mathbb{P}^N_\gamma$, however each condition in $\mathbb{P}^N_\gamma$ is compatible with some $p_n\restriction \gamma$. We can refine $\{\overline{p}_n \restriction \gamma|\; n<\omega\}$ to an antichain in $\mathbb{P}^N_\gamma$, and this antichain will remain maximal
		in $\mathbb{P}_\gamma$ by the induction hypothesis. Therefore $\{\overline{p}_n \restriction \gamma|\; n<\omega\}$ meets every condition from $\mathbb{P}_\gamma$, and in particular some $\overline{p}_n \restriction \gamma$ is compatible with $\overline{p}$ in $\mathbb{P}_\gamma$. But then $\overline{p}_n$ is compatible with $\overline{p}$ in $\mathbb{P}_\alpha$. $\qedhere$
	\end{itemize}
	
\end{proof}

\begin{proof}[Proof of Theorem \ref{mainthm2}]
	We start with any model $V\models \zfc + ``2^\omega=\kappa" + \slicekappa$, for example an extension of a model of $\ch$ by $\mathbb{C}_\kappa$ (see Corollary \ref{cohen}). Let $\mathbb{P}=\mathbb{P}_\theta$ be the forcing defined in the beginning of the Section. Suppose that a sequence $\{M_\alpha|\; \alpha<\kappa\}$ witnesses $\slicekappa$ in $V$, and $G\subseteq \mathbb{P}$ is generic over $V$. We aim to shows that the sequence $V[G\cap\mathbb{P}^{M_\alpha}]$ witnesses $\slicekappa$ in $V[G]$. For this we need to show two things
	\begin{enumerate}
		\item If $F \in 2^{<\kappa}\cap V[G]$, then $F \in V[G\cap\mathbb{P}^{M_\alpha}]$ for some $\alpha<\kappa$.
		\item None of the models $V[G\cap\mathbb{P}^{M_\alpha}]$ contains all reals from $V[G]$.
	\end{enumerate}
	
	\par Concerning $1.$ assume that $\mathbb{P}_\theta\Vdash \dot{F}\in 2^\delta$, for some ordinal $\delta<\kappa$. Without loss of generality $\delta=|\delta|\ge \lambda$. We can also assume that 
	
	$$\dot{F}=\{(\alpha,\alpha_n,p_n^\alpha)|\; \alpha<\delta,\; n<\omega\},$$
	and all conditions $p_n^\alpha$ are simple. We fix some elementary submodel $\overline{M}\prec \operatorname{H}((2^\theta)^+)$ of size $\delta$, of which we assume that $\delta+1\subseteq \overline{M}$, and $\{\sigma_\alpha|\; \alpha<\theta\}$, $\dot{F} \in \overline{M}$. Notice that $\delta+1 \subseteq \overline{M}$ guarentees that each of the conditions $p_n^\alpha$ is in $\overline{M}$. We pick $\alpha<\kappa$ big enough, so that $\overline{M}\cap \theta^\lambda \subseteq M_\alpha$. Now Lemma \ref{generalconmpletesuborders} shows that $p_n^\alpha \in \mathbb{P}_\alpha^n$, for all $\alpha<\delta$, $n<\omega$. This shows that $\dot{F}$ is a $\mathbb{P}_\theta^{M_\alpha}$-name, and it follows that 
	$$\dot{F}[G]=\dot{F}[G\cap \mathbb{P}^{M_\alpha}] \in V[G\cap\mathbb{P}^{M_\alpha}].$$
	
	\par Concerning $2.$ fix a real $r \in 2^{\omega} \setminus M_\alpha$. Let $\mathbb{C}_r$ be any representation of $\mathbb{C}$, from which the real $r$ is definable, and $\mathbb{C}_r$ is of the form
	$$\mathbb{C}_r=(\lambda,\le_r).$$
	This of course leaves plenty of space for what specifically $\mathbb{C}_r$ might be, but for the sake of concreteness we can define $\le_r$ as the transitive closure of the union of the following three relations:
	\begin{enumerate}
		\item $\le_r \restriction \omega\times \omega$ isomorphic to the countable atomless Boolean algebra,
		\item $\forall \; 1\le\alpha<\omega \quad \alpha \cdot \omega<_r \alpha \cdot \omega+1 \iff r(\alpha)=1$,
		\item $\forall \; 1\le\alpha<\omega \quad \alpha\cdot \omega>_r\alpha\cdot \omega +1 \iff r(\alpha)=0$.
	\end{enumerate}
	
	We pick $\gamma < \theta$ for which
	$$\mathbb{P}_\gamma \Vdash \dot{\mathbb{Q}}_\gamma=\mathbb{C}_r.$$
	In this case, we also have
	$$\mathbb{P}^{M_\alpha}_\gamma \Vdash \dot{\mathbb{Q}}^{M_\alpha}_\gamma=\{0\},$$
	since $r \notin M_\alpha$.
	
	Let $c$ be the real added by $\dot{\mathbb{Q}}_\gamma$ over $V[G\cap \mathbb{P}_\gamma]$. We claim that $c \notin V[G\cap \mathbb{P}^{M_\alpha}_\theta]$. Suppose towards a contradiction that $c \in V[G\cap \mathbb{P}^{M_\alpha}_\theta]$. Let $\mathbb{R}=\mathbb{P}_\theta^{M_\alpha}/(G\cap \mathbb{P}^{M_\alpha}_{\gamma+1})$. If $c \in V[G\cap \mathbb{P}^{M_\alpha}_\theta]$, then either $c$ is already in the initial fragment $V[G\cap \mathbb{P}^{M_\alpha}_{\gamma+1}]$, or is added generically over it by $\mathbb{R}$. The first option is ruled out, because 
	$$V[G\cap \mathbb{P}^{M_\alpha}_{\gamma+1}]=V[G\cap \mathbb{P}^{M_\alpha}_\gamma]\subseteq V[G\cap \mathbb{P}_\gamma],$$
	and $c$ was added generically over that model. Therefore in $V[G\cap \mathbb{P}^{M_\alpha}_{\gamma+1}]$ there is a $\mathbb{R}$-name $\dot{x}$ such that two of the following hold:
	\begin{itemize}
		\item $\mathbb{R} \Vdash ``\text{$\dot{x}$ is a new real}"$,
		\item $\dot{x}[G\cap \mathbb{P}_\theta^{M_\alpha}]=c.$
	\end{itemize}
	
	Consider now the forcing $\mathbb{R}$ in $V[G\cap \mathbb{P}_{\gamma+1}]$. It is still true that $\mathbb{R}$ forces $\dot{x}$ to be a new real, since this is an absolute property of the forcing and the name (it says that for any $r \in \mathbb{R}$ there is $n<\omega$, such that $r$ does not decide $\dot{x}(n)$). But this is a contradiction, since $c=\dot{x}[G] \in V[G\cap \mathbb{P}_{\gamma+1}]$.
	
	A standard verification shows that $V[G]\models ``2^\omega=\theta"$. This concludes the proof of Theorem \ref{mainthm2}.$\qedhere$
\end{proof}

\begin{cor}
	The following theories are consistent
	$$\zfc+\ma_{\omega_1}+\slicetwo+``2^\omega=\omega_2",$$
	$$\zfc+\ma_{\omega_1}+\slicetwo+``2^\omega=\omega_3",$$
	$$\zfc+\ma_{\omega_2}+\slicethree+``2^\omega=\omega_{29}".$$
\end{cor}

\section{Final comments}

We proved that $\ma_{\omega_1}$ and $\slice$ are not compatible. It looks reasonable to expect that for any regular cardinal $\kappa$
$$\ma_\kappa \implies \lnot \slicekappa.$$

We would like to thank prof. Piotr Zakrzewski for giving the idea of looking at the cardinal invariant $\sym$, and the anonymous referees for pointing out many flaws in earlier versions of the paper.

\end{document}